\documentclass[leqno]{amsart}
\usepackage{amsthm}
\usepackage{url}
\usepackage{tikz}

\numberwithin{equation}{section}
\numberwithin{figure}{section}

\theoremstyle{plain}
 \newtheorem{thm}{Theorem}
 \newtheorem{lem}[thm]{Lemma}

\theoremstyle{definition}
 
\theoremstyle{remark}

\newcommand{\Inn}{\mathrm{Inn}}
\newcommand{\rdiv}{/}
\newcommand{\ldiv}{\backslash}
\newcommand{\reltheta}{\mathrel{\theta}}

\title{Nonassociative right hoops}

\author[P. Jipsen]{Peter Jipsen}
\email{jipsen@chapman.edu}
\address{Faculty of Mathematics, School of Computational Sciences, Chapman University, 545 W. Palm Ave, Orange, CA 92866 USA}

\author[M. Kinyon]{Michael Kinyon$^\dagger$}
\thanks{${}^\dagger$ Partially supported by Simons Foundation Collaboration Grant 359872}
\email{mkinyon@du.edu}
\address{Department of Mathematics, University of Denver, Denver, CO 80208 USA}



\begin{document}

\begin{abstract}
The class of nonassociative right hoops, or narhoops for short, is defined
as a subclass of right-residuated magmas, and is shown to be a variety.
These algebras generalize both right quasigroups and right hoops,
and we characterize the subvarieties in which the operation $x\sqcap y=(x\rdiv y)y$
is associative and/or commutative. Narhoops with a left unit are proved to have a top element if and only if $\sqcap$ is commutative, and their congruences are
determined by the equivalence class of the left unit. We also show that the
four identities defining narhoops are independent.
\end{abstract}

\maketitle

\section{Introduction}

A \emph{residuated magma} is a partially ordered algebra $(A,\leq,\cdot,\rdiv,\ldiv)$
such that $(A,\leq)$ is a poset, $\cdot$ is a binary operation and $\rdiv,\ldiv$ are the right and left residuals of $\cdot$, which means the residuation property
\[
x\cdot y\leq z\quad\iff\quad x\leq z\rdiv y\quad\iff\quad y\leq x\ldiv z
\]
holds for all $x,y,z\in A$. As usual, we abbreviate $x\cdot y$ by $xy$
and adopt the convention that $\cdot$ binds more strongly than $\rdiv,\ldiv$.
A \emph{right-residuated magma} is of the form $(A,\le,\cdot,\rdiv)$
such that $(A,\le)$ is a poset and only the first equivalence is assumed, i.e.,
\[
\tag{Rres} xy\leq z\quad\iff\quad x\leq z\rdiv y\,.
\]
It is well-known that a partially ordered algebra $(A,\cdot,\rdiv)$ is a right-residuated magma if and only if, for all $x,y,z\in A$, the following hold:
\begin{align*}
  (x\rdiv y)y   &\leq       x\leq xy\rdiv y\,,       \\ 
  x\leq y       &\implies   xz\leq yz\,,             \\ 
  x\leq y       &\implies   x\rdiv z\leq y\rdiv z\,.  
\end{align*}

Define the term $x\sqcap y=(x\rdiv y)y$ and consider the following two varieties:
\begin{itemize}
\item A \emph{right quasigroup} is an algebra $(A,\cdot,\rdiv)$ satisfying the identities
$x\sqcap y = x = (xy)\rdiv y$. Right quasigroups are precisely those right-residuated magmas for which
the partial order $\leq$ is the equality relation.

\item A \emph{right hoop} is an algebra
$(A,\cdot,\rdiv)$ satisfying the identities $x\sqcap y=y\sqcap x$, $(x\rdiv x)y = y$,
and $x\rdiv (yz) = (x\rdiv z)\rdiv y$. Then it turns out that $x\rdiv x$ is a constant (denoted by $1$), the operation $\cdot$ is associative, the operation $\sqcap$ is a semilattice
operation, $1$ is the top element with respect to the semilattice order $\leq$, and $\rdiv$ is the right residual of $\cdot$ with respect to $\leq$.
\end{itemize}

Right hoops were introduced by Bosbach \cite{Bosbach1969,Bosbach1970} under the name ``left complementary semigroups''. He defined right complementary semigroups $(A,\cdot,\ldiv)$ dually and he called an algebra $(A,\cdot,\ldiv,\rdiv)$ a complementary semigroup if $(A,\cdot,\rdiv)$ is left complementary, $(A,\cdot,\ldiv)$ is right complementary, and the identity $(x\rdiv y)y = y(y\ldiv x)$ holds. Later, B\"{u}chi and Owens \cite{BuchiOwens} studied the case where $\cdot$ is commutative and $x\rdiv y = y\ldiv x$ holds, referring to these structures as ``hoops''. In the case where $\cdot$ is not necessarily commutative, Bosbach's complementary semigroups have been called ``pseudohoops'' (e.g., \cite{BDK}) or ``generalized hoops'' \cite{JipsenMontagna}. Bosbach \cite{Bosbach2009,Bosbach2010} later adapted the term ``hoop'' itself to the noncommutative case, suggesting that the case of B\"{u}chi and Owens be explicitly referred to as commutative or abelian.

Returning to the one-sided case, Rump \cite{Rump2008} renamed Bosbach's left complementary semigroups as ``left hoops'' (and used the opposite operation of our $\rdiv$). Bosbach followed suit in \cite{Bosbach2009,Bosbach2010}. In \cite{JipsenMontagna}, they were called ``right generalized hoops''. Here we follow Rump and Bosbach, but we reverse the left-right convention since $\rdiv$ is the \emph{right} residual of $\cdot$.

Note that the partial order is definable in both right quasigroups and right hoops, which motivates the following definition. A \emph{nonassociative right hoop} $(A,\leq,\cdot,\rdiv)$, or \emph{narhoop} for short, is a right-residuated magma such that for all $x,y\in A$
\[
\tag{N'} x\leq y   \quad\iff\quad  x\sqcap y = x = y\sqcap x\,.
\]
In any right-residuated magma $(x\rdiv y)y\le x$ or equivalently $x\sqcap y\leq x$ holds for all $x,y$, hence in a narhoop (N') implies that the identity $(x\sqcap y)\sqcap x=x\sqcap y$ holds. This provides an alternative definition for narhoops: they are right-residuated magmas that satisfy the identity
\[
\tag{N1} (x\sqcap y)\sqcap x=x\sqcap y
\]
and the bi-implication
\[
\tag{N} x\leq y \quad\iff\quad x = y\sqcap x
\]
since in the presence of (N1), if $x=y\sqcap x$ then multiplying by $y$ on the right we have $x\sqcap y=(y\sqcap x)\sqcap y=y\sqcap x=x$.

A \emph{nonassociative left hoop} or \emph{nalhoop} $(A,\leq,\cdot,\ldiv)$ is defined
dually and a \emph{nonassociative hoop} or \emph{nahoop} is both a narhoop and a nalhoop. Here we consider only narhoops and save the two-sided case for future research.

The two motivating varieties fit into this framework as follows.
\begin{itemize}
  \item A narhoop $(A,\cdot,\rdiv)$ is a right quasigroup if and only if $\leq$ is the equality relation.
  \item A narhoop $(A,\cdot,\rdiv)$ is a right hoop if and only if the identity
  $x/yz$ = $(x/z)/y$ and the
  quasiequation $x\sqcap y=x \implies x\leq y$ hold.
\end{itemize}

The main result of this section is that narhoops form a finitely based variety of algebras. To reduce the need for parentheses,
we assume that $x\rdiv y$ binds stronger than $x\sqcap y=(x\rdiv y)y$.

\begin{thm}\label{Thm:variety}
Let $(A,\leq,\cdot,\rdiv)$ be a narhoop. Then the following identities hold:
\begin{align*}
\tag{N1}  (x\sqcap y)\sqcap x &= x\sqcap y\,, \\
\tag{N2}  x &\leq xy\rdiv y\,, \\
\tag{N3}  (x\sqcap y)z &\leq xz\,, \\
\tag{N4} (x\sqcap y)\rdiv z &\leq x\rdiv z\,.
\end{align*}

Conversely, let $(A,\cdot,\rdiv)$ be an algebra with two binary operations satisfying \textup{(N1)--(N4)}, where
the relation $\leq$ is defined by \textup{(N)}. Then the identities
\begin{align*}
\tag{N5} x\sqcap xy/y &= x\,, \\
\tag{N6} (x\sqcap y)\rdiv y &= x\rdiv y\,, \\
\tag{N7} (x\sqcap y)\sqcap y &= x\sqcap y
\end{align*}
hold and $(A,\leq,\cdot,\rdiv)$ is a narhoop.
\end{thm}
\begin{proof}
Assume $(A,\leq,\cdot,\rdiv)$ is a narhoop. As noted above, the identity (N1) holds in narhoops. Every right-residuated magma satisfies (N2). In right-residuated magmas, both $\cdot$ and $\rdiv$ are order-preserving in their first argument. Since $x\sqcap y\leq x$, we get (N3) and (N4).

For the converse, suppose $(A,\cdot,\rdiv)$ satisfies (N1)--(N4), and $\leq$ is defined by (N).
From (N2), (N1) and (N2) again, we get (N5):
\[
x\sqcap (xy\rdiv y) = (xy\rdiv y\sqcap x)\sqcap (xy\rdiv y) = (xy\rdiv y)\sqcap x = x\,.
\]
For (N6), replace $x$ in (N5) by $x\rdiv y$ to get $x\rdiv y\sqcap(x\sqcap y)\rdiv y=x\rdiv y$ and then use (N4). To prove (N7) multiply (N6) on the right by $y$.

Now reflexivity of $\leq$ follows from (N5) and (N1): $x\sqcap x = (x\sqcap xy\rdiv y)\sqcap x = x\sqcap xy\rdiv y = x$.

For antisymmetry, if $x\leq y$ and $y\leq x$, then $x = y\sqcap x$ and $y = x\sqcap y$, hence $x = y\sqcap x = (y\sqcap x)\sqcap y = x\sqcap y = y$ using (N1) in the second equality.

Transitivity requires a bit more work. Suppose $x\leq y$ and $y\leq z$ so that
$x = y\sqcap x$ and $y = z\sqcap y$. First, note that
\[
z\rdiv x\sqcap y\rdiv x = z\rdiv x\sqcap (z\sqcap y)\rdiv x = (z\sqcap y)\rdiv x = y\rdiv x
\]
using (N4) and (N) in the second equality. Now we compute
\begin{align*}
  z\sqcap x    &= (z\sqcap x)\sqcap x &\text{by (N7)}\\
   &= (z\sqcap x)\sqcap (y\sqcap x) \\
   &= (z\rdiv x)x\sqcap (y\rdiv x)x \\
   &= (z\rdiv x)x\sqcap (z\rdiv x\sqcap y\rdiv x)x &\text{since $z\rdiv x\sqcap y\rdiv x = y\rdiv x$} \\
   &= (z\rdiv x \sqcap y\rdiv x)x &\text{by (N3) and (N)}\\
   &= (z\rdiv x \sqcap (z\sqcap y)\rdiv x)x \\
   &= ((z\sqcap y)\rdiv x)x &\text{by (N4) and (N)} \\
   &= (z\sqcap y)\sqcap x = y\sqcap x= x\,.
\end{align*}
From $x=z\sqcap x$ we get $x\leq z$ from (N).

Finally, we prove $\rdiv$ is the right residual of $\cdot$ with respect to $\leq$.
To do this, we verify $(x\rdiv y)y\leq x\leq xy\rdiv y$ and that $x\leq y$ implies $xz\leq yz$ and $x\rdiv z\leq y\rdiv z$ since the right residuation property is equivalent to these (quasi)identities.
Note that (N2) and (N) show $x\leq xy\rdiv y$. If $x\leq y$, then (N3) gives
\[
yz\sqcap xz = yz\sqcap (y\sqcap x)z = (y\sqcap x)z = xz,
\]
and so (N) implies $xz\leq yz$. By the same argument, (N4) gives $x\rdiv z\leq y\rdiv z$.

To prove $(x\rdiv y)y\leq x$, or equivalently $x\sqcap y\leq x$, substitute $x\rdiv x$ for $x$, $x$ for $y$,
and $(x\sqcap y)\rdiv x$ for $z$ in (N3) to get
\[
(x\rdiv x)x\sqcap(x\rdiv x\sqcap(x\sqcap y)\rdiv x)x=(x\rdiv x\sqcap(x\sqcap y)\rdiv x)x.
\]
Using (N4) this simplifies to $(x\sqcap x)\sqcap((x\sqcap y)\rdiv x)x=((x\sqcap y)\rdiv x)x$, so by
(N1), (N) and reflexivity we have $x\sqcap y\leq x\sqcap x=x$.
\end{proof}

The equational basis (N1)--(N4) for narhoops is independent as can be seen from  algebras $A_i=\{0,1\}$ ($i=1,2,3,4$) that each satisfy the axioms except for (N$i$).
\begin{itemize}
\item In $A_1$, $\cdot$ is ordinary multiplication and $x\rdiv y=y$.
\item In $A_2$, $x\cdot y=x$ and $x\rdiv y=1$.
\item In $A_3$, $x\cdot y$ is addition modulo 2 and $x\rdiv y = 0$ except that $1\rdiv 0 = 1$.
\item In $A_4$, $x\cdot y$ is the $\max$ operation and $x\rdiv y$ is addition modulo 2.
\end{itemize}

\section{Associativity and commutativity}
\label{Sec:Associativity}

In general, neither $\cdot$ nor the term operation $\sqcap$ of a narhoop is associative.
However $\sqcap$ is associative both in right quasigroups and in right hoops.
In right quasigroups, this follows from the identity $x\sqcap y = x$.
In right hoops, $\sqcap$ turns out be a semilattice operation (\cite{JipsenMontagna}, Lem. 4). In both cases the reduct $(A,\sqcap)$ is a \emph{left normal band}, that is, an idempotent semigroup satisfying the identity $x\sqcap y\sqcap z = x\sqcap z\sqcap y$.
In the absence of associativity we consider the following identities:
\begin{align*}
\tag{LN} (x\sqcap y)\sqcap z &= (x\sqcap z)\sqcap y \quad \text{the \emph{left normal identity}} \\
\tag{N8} x\sqcap(y\sqcap x) &=x\sqcap y  \\
\tag{N9} (x\sqcap(y\sqcap z))\sqcap z &=x\sqcap(y\sqcap z)\,.
\end{align*}

If $(A,\cdot,\rdiv)$ is a narhoop and $B\subseteq A$ is closed under $\sqcap$, then
$B$ inherits the order $\leq$ from $A$. We state the next two results in the slightly more general context of such subsets because we will need them in Theorem \ref{Thm:principalideals}.

\begin{thm}\label{Thm:lnb}
Let $(A,\cdot,\rdiv)$ be a right-residuated magma satisfying \textup{(N)} and let $B\subseteq A$ be closed under $\sqcap$. The following are equivalent.
\begin{enumerate}
  \item\quad $(B,\sqcap)$ is a left normal band;
  \item\quad $(B,\sqcap)$ is a semigroup;
  \item\quad $(B,\sqcap)$ satisfies \textup{(N8)} and \textup{(N9)}.
\end{enumerate}
\end{thm}
\begin{proof}
(1)$\implies$(2) is trivial.

(2)$\implies$(3): For $x,y,z\in B$, (N9) follows from associativity and idempotence. Now fix $x,y\in B$. We already have $(x\sqcap y)\sqcap x\leq x\sqcap y$ in any right-residuated magma. Associativity and idempotence gives $((x\sqcap y)\sqcap x)\sqcap (x\sqcap y) = x\sqcap y$, and so $x\sqcap y\leq (x\sqcap y)\sqcap x$ by (N). Therefore (N1) holds in $(B,\sqcap)$. Now (N8) in $(B,\sqcap)$ follows from (N1) and associativity.

(3)$\implies$(1): First note that for $x,y\in B$, we have
\[
(x\sqcap y)\sqcap x = (x\sqcap (y\sqcap x))\sqcap x = x\sqcap (y\sqcap x) = x\sqcap y\,,
\]
using (N8), (N9) and (N8) again. Thus (N1) holds in $(B,\sqcap)$. Now for $x,y,z\in B$, we compute
\begin{align*}
  (x\sqcap y)\sqcap z &= [(x\sqcap y)\sqcap z]\sqcap (x\sqcap y)    &&\text{by (N1)}\\
  &= ([(x\sqcap y)\sqcap z]\sqcap (x\sqcap y))\sqcap y              &&\text{by (N9)}\\
  &= [(x\sqcap y)\sqcap z]\sqcap y                                  &&\text{by (N1)}\\
  &\leq (x\sqcap z)\sqcap y\,,
\end{align*}
where the last inequality follows from $x\sqcap y\leq x$ and from $\sqcap$ being order-preserving in its first argument. Reversing the roles of $y$ and $z$, we also have $(x\sqcap z)\sqcap y\leq (x\sqcap y)\sqcap z$, and so (LN) holds.

Now for $x,y,z\in B$, we compute
\begin{align*}
  (x\sqcap y)\sqcap z &= (x\sqcap z)\sqcap y             &&\text{by (LN)}\\
   &= (x\sqcap z)\sqcap (y\sqcap (x\sqcap z))            &&\text{by (N8)}\\
   &= (x\sqcap z)\sqcap ((y\sqcap (x\sqcap z))\sqcap z)  &&\text{by (N9)}\\
   &= (x\sqcap z)\sqcap ((y\sqcap z)\sqcap (x\sqcap z))  &&\text{by (LN)}\\
   &= (x\sqcap z)\sqcap (y\sqcap z)                      &&\text{by (N8)}\\
   &= (x\sqcap (y\sqcap z))\sqcap z                      &&\text{by (LN)}\\
   &= x\sqcap (y\sqcap z)                                &&\text{by (N9)}\,.
\end{align*}
This proves associativity, and so $(B,\sqcap)$ is a left normal band as claimed.
\end{proof}

As noted, $\sqcap$-reducts of right hoops are semilattices. A natural nonassociative
generalization of right hoops is the quasivariety of right-residuated magmas described in the next result. The description of the $\sqcap$-reduct generalizes (\cite{JipsenMontagna}, Lem. 4).

\begin{thm}\label{Thm:comm_meet}
Let $(A,\cdot,\rdiv)$ be a right-residuated magma satisfying \textup{(N)}, and let $B\subseteq A$ be closed under $\sqcap$. The following are equivalent.
\begin{enumerate}
    \item\quad $(B,\sqcap)$ is a semilattice;
    \item\quad $(B,\sqcap)$ is commutative;
    \item\quad $(B,\sqcap)$ satisfies \textup{(N1)} and $x\sqcap y \leq y$ for all $x,y\in B$.
\end{enumerate}
\end{thm}
\begin{proof}
(1)$\implies$(2) is trivial.

(2)$\implies$(3): For $x,y\in B$, $(x\sqcap y)\sqcap x = x\sqcap (x\sqcap y)
= x\sqcap y$ since $x\sqcap y\leq x$ in any right-residuated magma. Thus (N1) holds,
and then using (N1), we have $y\sqcap (x\sqcap y) = (y\sqcap x)\sqcap x = y\sqcap x = x\sqcap y$, so that $x\sqcap y\leq y$ by (N).

(3)$\implies$(2): By (N1) and order-preservation of $\sqcap$ in the first argument,
$x\sqcap y = (x\sqcap y)\sqcap x\leq y\sqcap x$ for all $x,y\in B$. The reverse inequality holds by interchanging $x$ and $y$.

(2),(3)$\implies$(1): For $x,y\in B$, we have $x\sqcap (y\sqcap x) = (x\sqcap y)\sqcap x = x\sqcap y$ by commutativity and (N1), so (N8) holds. Since $u\sqcap v\leq u$, we have for $x,y,z\in B$, $(z\sqcap y)\sqcap x \leq z\sqcap y\leq z$. Thus by (N),
$z\sqcap ((z\sqcap y)\sqcap x) = (z\sqcap y)\sqcap x$. Rearranging both sides of this last equation using commutativity gives (N9). Applying Theorem \ref{Thm:lnb}, we infer that $(B,\sqcap)$ is associative, and since it is also commutative and idempotent, we have the desired result.
\end{proof}

In a left normal band, the identity $x\sqcap y\sqcap z = x\sqcap z\sqcap y$ essentially expresses the fact that every principal order ideal $(a] = \{x\in A\mid x\leq a\} = \{ a\sqcap x\mid x\in A\}$ is a subsemilattice. The same role is played by $(x\sqcap y)\sqcap z=(x\sqcap z)\sqcap y$ in narhoops. Note that $(a]=\{a\sqcap x\mid x\in A\}$ also holds for narhoops since
\[
x\leq a\quad\iff\quad x=a\sqcap x\quad\iff\quad x=a\sqcap y\text{ for some }y\in A
\]
where the second equivalence follows from $a\sqcap(a\sqcap x)=a\sqcap x$, which
in turn follows from (N) and $a\sqcap x\leq a$.

\begin{thm}\label{Thm:principalideals}
Let $(A,\cdot,\rdiv)$ be a right-residuated magma satisfying \textup{(N)}. Then:
\begin{enumerate}
\item Each $(a]$ is closed under $\sqcap$.
\item The following are equivalent.
\begin{enumerate}
\item $A$ is a narhoop,
\item each $((a],\sqcap)$ is commutative,
\item each $((a],\sqcap)$ is associative.
\end{enumerate}
\end{enumerate}
\end{thm}
\begin{proof}
(1) We have $(x\sqcap y)\sqcap z \leq x\sqcap y \leq x$ so $x\sqcap ((x\sqcap y)\sqcap z) = (x\sqcap y)\sqcap z$ for all $x,y,z\in A$. Thus
$x\sqcap ((x\sqcap y)\sqcap (x\sqcap z)) = (x\sqcap y)\sqcap (x\sqcap z)$, and letting $x=a$ we get the desired result.

(2) (a)$\implies$(b): From $x\sqcap z\leq x$ we obtain $(x\sqcap z)\sqcap (x\sqcap y)\leq x\sqcap (x\sqcap y) = x\sqcap(x\sqcap y) = x\sqcap y$. Letting $x=a$ again, it follows from Theorem~\ref{Thm:comm_meet} that $((a],\sqcap]$ is a semilattice.

(b)$\implies$(c) follows from (1) and Theorem \ref{Thm:comm_meet}.

For (c)$\implies$(a), it is enough to prove (N1).
Since $x \sqcap y\leq x$, we have $x \sqcap (x \sqcap y) = x \sqcap y$. Thus
\begin{align*}
x \sqcap y &= (x \sqcap y) \sqcap (x \sqcap y) &&\text{by idempotence,}\\
&= (x \sqcap y) \sqcap (x \sqcap (x \sqcap y))\,, &&\\
&= (x \sqcap y) \sqcap ((x \sqcap x) \sqcap (x \sqcap y)) &&\text{by idempotence,}\\
&= ((x \sqcap y) \sqcap (x \sqcap x)) \sqcap (x \sqcap y) &&\text{by associativity in }((x],\sqcap)\,,\\
&= ((x \sqcap y) \sqcap x) \sqcap (x \sqcap y) &&\text{by idempotence.}
\end{align*}
Hence $x \sqcap y \leq (x \sqcap y) \sqcap x \leq x \sqcap y$, so we get (N1).
\end{proof}

\section{Unital narhoops}
\label{Sec:Unital}

We now consider right-residuated magmas that satisfy (N) and have a left identity element.

\begin{lem}\label{Lem:preunital}
Let $(A,\cdot,\rdiv)$ be a right-residuated magma such that
\textup{(N)} holds. Then
\begin{enumerate}
\item $x\rdiv x$ is a maximal element for all $x\in A$,
\item the identity $(x\rdiv x)y\rdiv y = x\rdiv x$ holds in $A$, and
\item if $A$ has a top element then the term $x\rdiv x$ is this top element.
\end{enumerate}
\end{lem}
\begin{proof}
(1) Assuming $x\rdiv x\leq y$ for some $x,y\in A$, it suffices to show that $y\leq x\rdiv x$. By right residuation $y\leq yx\rdiv x$, hence $x\rdiv x\leq yx\rdiv x$ follows from transitivity of $\le$. An application of (Rres) gives $(x\rdiv x)x\leq yx$, and by reflexivity of $\le$ we have $x=x\sqcap x=(x\rdiv x)x$. Therefore $x\leq yx$, or equivalently $x=yx\sqcap x=(yx\rdiv x)x$. Using (Rres) once more we have $y\leq yx\rdiv x\leq x\rdiv x$.

(2) From $(x\rdiv x)y\leq (x\rdiv x)y$ we obtain $x\rdiv x\leq ((x\rdiv x)y)/y$, so equality follows from (1).

(3) is also an obvious consequence of (1).
\end{proof}

Note that the preceding lemma applies to narhoops.

\begin{lem}\label{Lem:unital}
Let $(A,\cdot,\rdiv)$ be a right-residuated magma that satisfies \textup{(N)}. The following are equivalent.
\begin{enumerate}
  \item $x\rdiv x = y\rdiv y$ for all $x,y\in A$;
  \item $(x\rdiv x)y = y$ for all $x,y\in A$;
  \item There exists $e\in A$ such that $ey=y$ for all $y\in A$.
\end{enumerate}
When these conditions hold, the element $1 = x\rdiv x$ is the maximum left identity element
in $(A,\leq)$.
\end{lem}
\begin{proof}
(1)$\implies$(2): $(x\rdiv x)y = (y\rdiv y)y = y\sqcap y = y$ by (N).

(2)$\implies$(1): Using Lemma \ref{Lem:preunital}(2), $x\rdiv x = (x\rdiv x)y\rdiv y = y\rdiv y$.

(1),(2)$\implies$(3): Setting $e = x\rdiv x$, we have $ey=y$ for all $y\in A$.

(3)$\implies$(2): Since $ex=x$, we have $e\leq x\rdiv x$ by (Rres), and thus
$y = ey \leq (x\rdiv x)y$ by order-preservation of $\cdot$ in the first argument. Hence $y = ((x\rdiv x)y\rdiv y)y = (x\rdiv x)y$ using Lemma \ref{Lem:preunital}(2).

The remaining assertion follows from $e\leq x\rdiv x$ as in the preceding paragraph.
\end{proof}

A right-residuated magma $(A,\cdot,\rdiv)$ that satisfies (N)
is called \emph{unital} if the equivalent conditions of
Lemma \ref{Lem:unital} hold. In this case as in the lemma, we denote by $1 = x\rdiv x$ the distinguished left identity element. Note that the lemma does not claim that $1$ is the unique left identity element, even when $\sqcap$ is commutative. Unital narhoops are defined similarly.

\begin{thm}\label{Thm:finite}
If $A$ is finite and unital, then $1$ is the unique left identity element.
\end{thm}
\begin{proof}
Let $e\in A$ be a left identity element. Define a sequence of elements of $A$ as follows:
$e_0 := e$, $e_k := e_{k-1}\rdiv 1$ for $k > 0$.

We claim that for all $k\geq 0$, $e_k \leq 1$. For $k = 0$, this follows
from Lemma \ref{Lem:unital}. If the claim holds for any $k\geq 0$, then
$e_{k+1} = e_k \rdiv 1 \leq 1\rdiv 1 = 1$,
using order-preservation of $/$ in the numerator. Thus the claim follows by induction on $k$. In particular, we have
\[
  e_{k+1} 1 = (e_k\rdiv 1)1 = e_k \sqcap 1 = e_k        \tag{$\ast$}
\]
for all $k\geq 0$.

Since $A$ is finite, the sequence $\{e_k\}_{k\geq 0}$ is finite.
Thus there exist $m > n \geq 0$ such that $e_m = e_n$. If $n > 0$, then by ($\ast$),
$e_{m-1} = e_m 1 = e_n 1 = e_{n-1}$. Thus we may assume $n = 0$, that is,
$e_m = e$ for some $m > 0$. Thus $e_{m-1} = e_m 1 = e 1 = 1$, and so
$e = e_m = e_{m-1}\rdiv 1 = 1\rdiv 1 = 1$. This completes the proof.
\end{proof}

It is an open problem whether there exists a (necessarily infinite) unital narhoop with more than one left identity element.

Note that in a unital narhoop $A$ the partial order $\leq$ can be characterized in terms of $1$ and $\rdiv$:
\[
x\leq y \quad\iff\quad y\rdiv x = 1     \tag{U}
\]
for all $x,y\in A$.
Furthermore, the left identity $1$ can also be used to characterize the
commutativity of $\sqcap$.

\begin{thm}
Let $(A,\cdot,\rdiv,1)$ be a unital narhoop.
Then:
\begin{enumerate}
\item The left unit $1$ is the top element of $(A,\leq)$ if and only if $\sqcap$ is commutative.
\item The principal ideal $(1]$ is a subnarhoop and $((1],\sqcap)$ is a semilattice.
\end{enumerate}
\end{thm}
\begin{proof}
(1) The forward direction follows from Theorem~\ref{Thm:principalideals}.

Conversely, assume $\sqcap$ is commutative. Then for all $x,y\in A$,
\begin{align*}
  x &\leq xy\rdiv y         &&\text{by (Rres)} \\
   &= (xy\rdiv y)y\rdiv y   &&\text{by (Rres)} \\
   &= (xy\sqcap y)\rdiv y     && \\
   &= (y\sqcap xy)\rdiv y    &&\text{by commutativity of }\sqcap \\
   &\leq y\rdiv y           &&\\
   &= 1\,,
\end{align*}
where the second-last step follows from $x\sqcap y\leq x$ and order-preservation
of $/$ in the numerator. Thus $1$ is the top element of $(A,\leq)$ as claimed.

(2) If $a,b\leq 1$, then $ab\leq 1b = b\leq 1$ and $a\rdiv b \leq 1\rdiv b = 1$
by order-preservation of $\cdot$ and $/$ in the first argument. Thus $(1]$ is a subnarhoop and the rest follows from Theorem \ref{Thm:principalideals}.
\end{proof}

We end this section by considering right-residuated magmas with a bottom or top element.

\begin{thm}
Let $(A,\cdot,\rdiv)$ be a right-residuated magma that satisfies \textup{(N)}.
\begin{enumerate}
\item If $A$ has a bottom element $0$, then $0\rdiv 0$ is the top element.
\item If $A$ has a top element then
\renewcommand\labelenumii{(\roman{enumii})}
\begin{enumerate}
\item $A$ is unital,
\item $x\sqcap y$ is a lower bound for $x$, $y$, and
\item if $A$ is finite then $A$ has a bottom element.
\end{enumerate}
\end{enumerate}
\end{thm}
\begin{proof}
(1) Assume $0\leq x$ holds for all $x\in A$. Then $0\leq x0$, hence $x0\sqcap 0=0$ by (N). It follows that $(x0\rdiv 0)0\leq 0$, so by residuation we get $x0\rdiv 0\leq 0\rdiv 0$. Since $x\leq x0\rdiv 0$, we obtain $x\leq 0\rdiv 0$.

(2)(i) Assume $\top$ is the top element of $A$. Then $\top x\leq\top$, and since $\rdiv $ is order-preserving in the left argument, we have $(\top x)\rdiv x\leq \top\rdiv x$. From $\top\leq (\top x)\rdiv x$ we see that $\top=(\top x)\rdiv x = \top\rdiv x$.
Now $x\leq \top$ implies $\top\sqcap x=x$ by (N), hence $(\top\rdiv x)x=x$. Since we already showed that $\top\rdiv x=\top$, it follows that $\top x=x$.

(2)(ii) It is always the case that $x\sqcap y\leq x$ since $x\rdiv y\leq x\rdiv y$
implies $(x\rdiv y)*y\leq x$ by residuation.
As shown in (i), $\top=(\top x)\rdiv x$ and $\top x=x$, so $\top=x\rdiv x$ for any $x\in A$. Hence $x\rdiv y\leq y\rdiv y$, and again by residuation we have $x\sqcap y=(x\rdiv y)y\leq y$.

(2)(iii) follows directly from (2)(ii).
\end{proof}

\section{Unital Congruences}
\label{Sec:congruences}
A congruence $\theta$ on a narhoop $(A,\cdot,\rdiv)$ will be said to be \emph{unital} if
the factor narhoop $A/\theta$ is unital. In other words, $\theta$ is unital if and only
if $x\rdiv x\reltheta y\rdiv y$ for all $x,y\in A$. If $A$ itself is unital then every
congruence on $A$ is unital. For a unital congruence on an arbitrary narhoop, set
\begin{align*}
  N_{\theta} &= \{x\in A\mid x\reltheta y\rdiv y\text{ for some }y\in A\} \\
   &= \{x\in A\mid x\reltheta y\rdiv y\text{ for all }y\in A\}\,,
\end{align*}
where the second equality follows since $\theta$ is unital.
Analogous to the relationship between congruences and normal subgroups in group theory,
we now show that $\theta$ is determined by the congruence class $N_{\theta}$.

\begin{lem}\label{Lem:congruence}
For all $x,y\in A$, $x\reltheta y$ if and only if $x\rdiv y,\ y\rdiv x\in N_{\theta}$.
\end{lem}
\begin{proof}
If $x\reltheta y$, then $x\rdiv y\reltheta y\rdiv y$ and similarly,
$y\rdiv x\reltheta x\rdiv x$, and so $x\rdiv y,\ y\rdiv x\in N_{\theta}$.

Conversely, suppose $x\rdiv y,\ y\rdiv x\in N_{\theta}$.
Denoting the $\theta$-classes with square brackets, we have $[x]\rdiv [y] = [x\rdiv y] = N_{\theta}$ and $[y]\rdiv [x] = [y\rdiv x] = N_{\theta}$ in $A/\theta$. Since $N_{\theta}$ is the distinguished left identity element of $A/\theta$, we obtain $[x]\leq [y]$ and $[y]\leq [x]$ using (U). Thus $[x]=[y]$, that is, $x\reltheta y$.
\end{proof}

In order to state our characterizations concisely, we introduce six families of mappings on a narhoop $(A,\cdot,\rdiv)$. For each $x,y\in A$, $i = 1,\ldots,6$, define
$\phi_{i,x,y} : A\to A$ by
\begin{align*}
  \phi_{1,x,y}(z) &= (zx\cdot y)\rdiv xy\,,         \\
  \phi_{2,x,y}(z) &= (zx\rdiv y)\rdiv (x\rdiv y)\,, \\
  \phi_{3,x,y}(z) &= (x\cdot zy)\rdiv xy\,,         \\
  \phi_{4,x,y}(z) &= (x\rdiv zy)\rdiv (x\rdiv y)\,, \\
  \phi_{5,x,y}(z) &= xy\rdiv (x\cdot zy)\,,         \\
  \phi_{6,x,y}(z) &= (x\rdiv y)\rdiv (x\rdiv zy)\,.
\end{align*}
Again keeping analogies with group theory in mind, let
\[
\Inn(A) = \langle \phi_{i,x,y}\mid i=1,\ldots,6,\ x,y\in A\rangle\,,
\]
that is, $\Inn(A)$ is the transformation semigroup on $A$ generated by these six families of mappings.

\begin{thm}\label{Thm:congruence}
Let $\theta$ be a unital congruence on a narhoop $(A,\cdot,\rdiv)$. Then:
\begin{enumerate}
  \item\quad $N_{\theta}$ is a subnarhoop of $A$;
  \item\quad For all $x,y\in A$, if $x\leq y$ and $x\in N_{\theta}$, then $y\in N_{\theta}$;
  \item\quad $N_{\theta}$ is invariant under $\Inn(A)$.
\end{enumerate}
\end{thm}
\begin{proof}
(1) For $x,y\in N_{\theta}$, $x\reltheta y\rdiv y$ and so $xy\reltheta (y\rdiv y)y = y\sqcap y = y$. Thus $xy\in N_{\theta}$. Also, $x\reltheta y$ and so $x\rdiv y\reltheta y\rdiv y$, that is, $x\rdiv y\in N_{\theta}$.

(2) Suppose $x\leq y$ and $x\in N_{\theta}$. For all $u\in A$, $x\reltheta u\rdiv u$ and so $xu\reltheta (u\rdiv u)u = u\sqcap u = u$. Now $xu\leq yu$ by order-preservation, and hence $yu = yu\sqcap u\reltheta yu\sqcap xu = xu\reltheta u$. Thus $yu\rdiv u\reltheta u\rdiv u\reltheta x$, and so $(yu\rdiv u)\sqcap y\reltheta x\sqcap y = x$ since $x\leq y$. We have $y\leq yu\rdiv u$ by residuation, hence $y = (yu\rdiv u)\sqcap y\reltheta x$.
Therefore $y\in N_{\theta}$ as claimed.

(3) Each $\phi_{i,x,y}(z)$ is a term in the operations $\cdot$ and $\rdiv$ and so if $z\reltheta w$, it follows that $\phi_{i,x,y}(z)\reltheta \phi_{i,x,y}(w)$. Thus it is enough to show $\phi_{i,x,y}(z\rdiv z)\in N_{\theta}$ for some $z\in A$. For example,
\[
\phi_{1,x,y}(x\rdiv x) = ((x\rdiv x)x\cdot y)\rdiv xy = (x\sqcap x) y\rdiv xy = xy\rdiv xy\in N_{\theta}\,.
\]
The other five cases are proven similarly.
\end{proof}

Let $(A,\cdot,\rdiv)$ be a narhoop. A nonempty subset $N$ of $A$ is said to be a
\emph{normal subnarhoop} of $A$, denoted $N\unlhd A$, if the following hold:
\begin{enumerate}
  \item $N$ is a subnarhoop of $A$;
  \item For all $x,y\in A$, if $x\leq y$ and $x\in N$, then $y\in N$;
  \item $N$ is invariant under $\Inn(Q)$.
\end{enumerate}

\begin{lem}\label{Lem:N_preorder}
Let $(A,\cdot,\rdiv)$ be a narhoop and assume $N\unlhd A$ is nonempty.
Define $\preceq_N$ on $A$ by $x\preceq_N y$ if and only if $y\rdiv x\in N$. Then:
\begin{enumerate}
  \item\quad For all $x,y,z\in A$, if $x\preceq_N y$, then $xz\preceq_N yz$;
  \item\quad For all $x,y,z\in A$, if $x\preceq_N y$, then $x\rdiv z\preceq_N y\rdiv z$;
  \item\quad $\preceq_N$ is a preorder.
\end{enumerate}
\end{lem}
\begin{proof}
(1), (2): Suppose $x\preceq_N y$, that is, $x\rdiv y\in N$. Then for all $z\in A$,
$\phi_{1,y,z}(x\rdiv y)\in N$ and $\phi_{2,y,z}(x\rdiv y)\in N$. We have
\[
\phi_{1,y,z}(x\rdiv y) = (x\sqcap y)z\rdiv yz \quad\text{and}\quad
\phi_{2,y,z}(x\rdiv y) = ((x\sqcap y)\rdiv z)\rdiv (y\rdiv z)\,.
\]
Now $x\sqcap y\leq x$ implies both $(x\sqcap y)z\rdiv yz\leq xz\rdiv yz$ and
$((x\sqcap y)\rdiv z)\rdiv (y\rdiv z)\leq (x\rdiv z)\rdiv (y\rdiv z)$ using
order-preservation of $\cdot$ and $/$ in the first argument. Thus $xz\rdiv yz\in N$ and $(x\rdiv z)\rdiv (y\rdiv z)\in N$. Therefore $xz\preceq_N yz$ and $x\rdiv z\preceq_N y\rdiv z$, which establishes both (1) and (2).

(3) Next we show $\preceq_N$ is transitive. Suppose $x\preceq_N y$ and $y\preceq_N z$, that is, $x\rdiv y, y\rdiv z\in N$. By part (2), $(x\rdiv z)\rdiv (y\rdiv z)\in N$. Since $N$ is a subnarhoop, $(x\rdiv z)\sqcap (y\rdiv z) = [(x\rdiv z)\rdiv (y\rdiv z)](y\rdiv z) \in N$. Finally, $(x\rdiv z)\sqcap (y\rdiv z)\leq x\rdiv z$ and so $x\rdiv z\in N$, that is, $x\preceq_N z$.

Finally we show $\preceq_N$ is reflexive, that is, $y\rdiv y\in N$ for all $y\in A$.
Since $N\neq \emptyset$, fix $x\in N$. For all $t,u,v,w\in A$,
$\phi_{4,t,u}(x)\cdot \phi_{3,v,w}(x)\in N$ since $N$ is invariant under $\Inn(Q)$
and a subnarhoop. For fixed $y\in A$, set
\[
t = (y\rdiv y)\cdot xy\,,\quad u = y\,,\quad v = y\rdiv y\,,\quad w = y\,.
\]
Then
\[
\phi_{4,t,u}(x) = [((y\rdiv y)\cdot xy)\rdiv xy]\rdiv [((y\rdiv y)\cdot xy)\rdiv y]
= (y\rdiv y)\rdiv [((y\rdiv y)\cdot xy)\rdiv y]\,,
\]
using Lemma~\ref{Lem:preunital}(2), and
\[
\phi_{3,v,w} = ((y\rdiv y)\cdot xy)\rdiv (y\rdiv y)y = ((y\rdiv y)\cdot xy)\rdiv y\,.
\]
Thus
\begin{align*}
\phi_{4,t,u}(x)\cdot \phi_{3,v,w}(x) &=
((y\rdiv y)\rdiv [((y\rdiv y)\cdot xy)\rdiv y])\cdot [((y\rdiv y)\cdot xy)\rdiv y] \\
&= (y\rdiv y)\sqcap [((y\rdiv y)\cdot xy)\rdiv y] \\
&\leq y\rdiv y\,.
\end{align*}
It follows that $y\rdiv y\in N$ for all $y\in A$, as claimed. This completes the proof.
\end{proof}

The next result is new even in the less general setting of right quasigroups with left identity elements.

\begin{thm}\label{Thm:normal}
Let $(A,\cdot,\rdiv)$ be a narhoop and assume $N\unlhd A$ is nonempty.
Define $\theta_N$ on $A$ by $x\reltheta_N y$ if and only if $x\rdiv y,y\rdiv x\in N$. Then $\theta_N$ is a unital congruence and $N_{\theta_N} = N$.
\end{thm}
\begin{proof}
We have $x\reltheta_N y$ if and only if $x\preceq_N y$ and $y\preceq_N x$, and so by
Lemma \ref{Lem:N_preorder}, $\theta_N$ is the equivalence relation induced by the preorder $\preceq_N$. To prove that $\theta_N$ is a congruence, it
remains to establish
\begin{itemize}
  \item[(a)]\quad $x\reltheta_N y$ implies $zx\reltheta_N zy$, and
  \item[(b)]\quad $x\reltheta_N y$ implies $z\rdiv x\reltheta_N z\rdiv y$
\end{itemize}
for all $x,y,z\in A$.

In preparation for both parts, suppose $x\reltheta_N y$, that is, $x\rdiv y, y\rdiv x\in N$. We have
\begin{alignat}{2}
  \phi_{2,y,z}(x\rdiv y) &= ((x\sqcap y)\rdiv z)\rdiv (y\rdiv z) &\ \in N\,, \label{Eq:phi2_x/y}\\
  \phi_{3,z,y}(x\rdiv y) &= z(x\sqcap y)\rdiv zy                 &\ \in N\,, \label{Eq:phi3_x/y}\\
  \phi_{4,z,y}(x\rdiv y) &= (z\rdiv (x\sqcap y))\rdiv (z\rdiv y) &\ \in N\,. \label{Eq:phi4_x/y}
\end{alignat}

In \eqref{Eq:phi2_x/y}, set $z = x$. Since $y\rdiv x\in N$ and $N$ is a subnarhoop,
we obtain
\[
[((x\sqcap y)\rdiv x)\rdiv (y\rdiv x)](y\rdiv x) = ((x\sqcap y)\rdiv x)\sqcap (y\rdiv x)\in N\,.
\]
Because $((x\sqcap y)\rdiv x)\sqcap (y\rdiv x) \leq (x\sqcap y)\rdiv x$, we have
\begin{equation}\label{Eq:congtmp0}
  (x\sqcap y)\rdiv x\in N\,.
\end{equation}

From \eqref{Eq:congtmp0}, we have
\begin{alignat}{4}
  \phi_{5,z,x}((x\sqcap y)\rdiv x)   &= zx\rdiv z[(x\sqcap y)\sqcap x]
                                    &&= zx\rdiv z(x\sqcap y)
                                    &\ \in N\,, \label{Eq:phi5_xy/x} \\
  \phi_{6,z,x}((x\sqcap y)\rdiv x)   &= (z\rdiv x)\rdiv (z\rdiv [(x\sqcap y)\sqcap x])
                                    &&= (z\rdiv x)\rdiv (z\rdiv (x\sqcap y))
                                    &\ \in N\,, \label{Eq:phi6_xy/x}
\end{alignat}
where we have used (N1) in both calculations.

Applying Lemma \ref{Lem:N_preorder} to \eqref{Eq:phi5_xy/x}, we have
\begin{equation}\label{Eq:congtmp2}
(zx\rdiv zy)\rdiv (z(x\sqcap y)\rdiv zy)\in N\,.
\end{equation}
Combining this with \eqref{Eq:phi3_x/y} and the fact that $N$ is a subnarhoop, we get
\[
[(zx\rdiv zy)\rdiv (z(x\sqcap y)\rdiv zy)](z(x\sqcap y)\rdiv zy) =
(zx\rdiv zy)\sqcap (z(x\sqcap y)\rdiv zy)\in N\,.
\]
Since $(zx\rdiv zy)\sqcap (z(x\sqcap y)\rdiv zy)\leq zx\rdiv zy$, we have $zx\rdiv zy\in N$, that is, $zy\preceq_N zx$. Reversing the roles of $x$ and $y$ (which makes sense because we assumed both $x\rdiv y,y\rdiv x\in N$), we also have $zx\preceq_N zy$.
Therefore $zx\reltheta_N zy$, which establishes (a).

Applying Lemma \ref{Lem:N_preorder} to \eqref{Eq:phi6_xy/x}, we have
\begin{equation}\label{Eq:congtmp1}
[(z\rdiv x)\rdiv (z\rdiv y)]\rdiv [(z\rdiv (x\sqcap y))\rdiv (z\rdiv y)] \in N\,.
\end{equation}
Combining this with \eqref{Eq:phi4_x/y} and the fact that $N$ is a subnarhoop, we get
\begin{multline*}
  ([(z\rdiv x)\rdiv (z\rdiv y)]\rdiv [(z\rdiv (x\sqcap y))\rdiv (z\rdiv y)])[(z\rdiv (x\sqcap y))\rdiv (z\rdiv y)] \\
  = [(z\rdiv x)\rdiv (z\rdiv y)]\sqcap [(z\rdiv (x\sqcap y))\rdiv (z\rdiv y)]\in N\,.
\end{multline*}
Since $[(z\rdiv x)\rdiv (z\rdiv y)]\sqcap [(z\rdiv (x\sqcap y))\rdiv (z\rdiv y)]
\leq (z\rdiv x)\rdiv (z\rdiv y)$, we have $(z\rdiv x)\rdiv (z\rdiv y)\in N$, that is,
$z\rdiv y\preceq_N z\rdiv x$. Reversing the roles of $x$ and $y$, we also have
$z\rdiv x\preceq_N z\rdiv y$. Therefore $z\rdiv x\reltheta z\rdiv y$, which
establishes (b)

\smallskip

We turn to the final assertion of the theorem.
If $x\in N$, then since each $y\rdiv y\in N$ (Lemma \ref{Lem:N_preorder}) and $N$ is a
subnarhoop, we have $x\rdiv (y\rdiv y), (y\rdiv y)\rdiv x\in N$. Thus $x\reltheta_N y\rdiv y$, that is, $x\in N_{\theta_N}$. Therefore $N\subseteq N_{\theta_N}$. Conversely, if $x\in N_{\theta_N}$, then $x\rdiv (y\rdiv y)\in N$. Again using $y\rdiv y\in N$ and $N$ being a subnarhoop, we have $x\sqcap (y\rdiv y) = (x\rdiv (y\rdiv y))(y\rdiv y)\in N$. This implies $x\in N$ because $x\sqcap (y\rdiv y)\leq x$. Therefore $N_{\theta_N}\subseteq N$. This completes the proof.
\end{proof}

\section*{Acknowledgments}

This research was supported by the automated theorem prover \textsc{Prover9} and the finite model builder \textsc{Mace4}, both created by McCune \cite{McCune}. We would like to thank Bob Veroff for hosting the 2016 Workshop on Automated Deduction and its Applications to Mathematics (ADAM) which is where our collaboration began.

\end{document}